\documentclass[11pt,a4paper,oneside]{amsart}
\usepackage[T1]{fontenc}

\usepackage{tikz}
\usepackage{latexsym,amsfonts,amsmath,amsthm,amssymb,mathtools,mathrsfs}

\usepackage{graphicx}
\usepackage{comment}
\usepackage{placeins}
\usepackage[dvipsnames]{xcolor}
\usepackage{soul}
\usepackage{subcaption}
\usepackage{pgfplots}

\usepackage{xurl}
\usepackage{microtype}

\usepackage[
    colorlinks=true,
    breaklinks=true,
    linkcolor=red,
    citecolor=ForestGreen,
    urlcolor=blue
]{hyperref}

\newcommand{\abs}[1]{\lvert#1\rvert}

\pgfplotsset{compat=newest}

\theoremstyle{plain}
\newtheorem{theorem}{Theorem}[section]
\newtheorem{lemma}[theorem]{Lemma}
\newtheorem{remark}[theorem]{Remark}

\newtheorem{prop}[theorem]{Proposition}

\title{
Effective Dynamics of Disclination Pairs
}

\author{Nicolò Briatico $^\dagger$}
\author{Pierluigi Cesana$^\ddagger$}
\author{Marco Morandotti $^\dagger$}
\thanks{$\dagger$ Dipartimento di Scienze Matematiche ``G.~L.~Lagrange'', Politecnico di Torino, Corso Duca degli Abruzzi, 24, 10129 Torino, Italy (\url{nicolo.briatico@polito.it}, \url{marco.morandotti@polito.it}).}
\thanks{$\ddagger$ Institute of Mathematics for Industry, Kyushu University, 744 Motooka, Fukuoka, 819-0395, Japan (\url{cesana@math.kyushu-u.ac.jp}).}

\date{\today}



\begin{document}

\begin{abstract}
The dissipative dynamics of a pair of wedge disclinations with opposite Frank angles confined to a circular domain is investigated under the assumption of radially symmetric motion. Two distinct dynamical regimes are identified: a diverging-disclination regime and an annihilating-dipole regime. The stationary points associated with both regimes are determined and fully characterized in terms of their stability, and quantitative estimates for the characteristic evolution times in their vicinity are derived. In the annihilating-dipole regime, it is shown that, after a suitable  time rescaling, the resulting  law of motion for the dipole coincides with that of an edge dislocation.
\end{abstract}

\maketitle

\tableofcontents

\section{Introduction}

Dislocations and disclinations (both referred to as defects in this paper, with some abuse of terminology) were introduced by Vito Volterra as prototypical singularities arising from the violation of geometric compatibility in elastic bodies,  in planar elasticity \cite{V07}. In continuum elasticity, disclinations correspond to violations of rotational symmetry and are characterized by their Frank angle, measuring the angular mismatch, and store large, although finite, elastic energies on bounded domains. Their associated stresses and strains are singular, yet square-integrable. Dislocations, on the other hand, are genuine topological singularities carrying a topological charge, and are characterized by the nano-scale Burgers vector. In the classical continuum description, they possess singular energies due to non-integrable stresses and strains.
From a mathematical viewpoint, dislocations share important similarities with Ginzburg--Landau vortices and with point charges in two-dimensional Coulomb gases, since all involve long-range logarithmic interactions between topological or effective charges \cite{alastuey1981classical,SandierSerfaty152D}.

Because of their distinct energetic and topological nature, the modeling of systems containing both dislocations and disclinations remains challenging.  Indeed, dislocations have infinite elastic energy in the continuum limit, while disclinations have finite but size-dependent energy, which complicates the superposition of their respective fields.
Nevertheless, Volterra’s construction shows that both dislocations and disclinations arise naturally as equilibrium configurations in planar linear elasticity under incompatible kinematics. Thus, they represent fundamental defect structures, both theoretically predicted and experimentally observed \cite{nn102598m, CPL14, ET67,I19,   Inamura01022013,INAMURA2017351,PL13,TE68}.

A key structural link between these defects was identified by Eshelby, who observed that the elastic field of an edge dislocation can be represented, at the level of kinematics, as that of a dipole of wedge disclinations, provided a precise relation holds between Burgers vector, Frank angle, and dipole geometry \cite{Eshelby66}. This observation reveals that dislocations can be viewed as a singular limiting configuration of disclination dipoles. Building on this idea, a variational framework coupling dislocations and disclinations was developed in \cite{CDLM24} within linearized elasticity, formulated in terms of the Airy stress function. In particular, it was shown that an edge dislocation can be obtained as the variational limit of a shrinking dipole of wedge disclinations, both at the level of minimizers and minimal energies in static regimes.

When investigating the effects that the presence of these defects has on the material properties, it is fundamental to study their motion, as it has been proved that dislocation motion is linked to plasticity and hardening phenomena \cite{Orowan1934,Polanyi1934, Taylor1934}.
 While dislocation dynamics has been extensively studied in both inertial and dissipative settings (see, \emph{e.g.}, \cite{BITZEK200411, 10.1007/978-1-4020-2111-4_5, SD2011, Tang2018}), 
comparatively little is known about the dynamics of disclinations. 
Recent work \cite{CGMP2025} introduced and analyzed a dissipative model for finite systems of wedge disclinations, obtaining existence results together with a description of collision times and dipole interactions. However, several fundamental questions remain open, including the analysis of disclination dipoles in their natural scaling regime, the connection with dislocation dynamics, and the renormalization of the associated time-dependent models. 
For related approaches  and models, see \cite{ACHARYA2001761, ACHARYA2010766, 10.1007/978-3-319-18573-6_5, GarroniVanMeursPeletierScardia19, VanMeursMorandotti19, vMPP2022, ZA2018, ZHANG2015145, BCH15, CH20}.

In this work, we study dissipative dynamics of a pair of wedge disclinations with opposite Frank angles. Such configurations naturally screen each other's elastic fields and exhibit two distinct behaviors: when the defects are sufficiently close, they attract and may annihilate, whereas for larger separations they drift toward the boundary and are eventually expelled. Previous work \cite{CGMP2025} considered mainly symmetric configurations in circular domains. Here we extend the analysis to asymmetric settings and investigate the fine structure of the dynamics.

After introducing the variational model based on the Airy stress function (Section~\ref{sec_mechanical}) and a dissipative gradient-flow structure (Section~\ref{sec_dissipative}), we study the evolution of a pair of disclinations under radial motion from general initial configurations (Section~\ref{analysis_dynamics}). We identify two regimes depending on the initial separation: a diverging regime and a converging regime.
In the diverging regime, the defects migrate toward the boundary and behave as isolated disclinations (Section~\ref{section diverging}).
 In the converging regime (Section~\ref{section converging}), corresponding to a disclination dipole, the dynamics exhibits a strong separation of time scales: the midpoint evolves slowly, while the inter-defect distance evolves rapidly.
 Exploiting this behavior, we derive a reduced effective dynamics, and show that it coincides, in the sense of Eshelby’s dipole construction, with the dynamics of an edge dislocation located at the center of the dipole (Section~\ref{sec_rescaling}).
  Finally, we complement the analytical results with numerical simulations and phase diagrams computed in MATLAB (Section~\ref{phase diagram}).

\section{Mechanical model}\label{sec_mechanical}
Let $B_R=B_R(0)\subset\mathbb{R}^2$ be the open disk of radius $R>0$ centered at the origin; we interpret it as the cross-section of an infinite cylinder containing line defects parallel to its axis.
Introducing the Young's modulus $E>0$, the Poisson's ratio $\nu\in(-1,\frac{1}{2})$, and the Airy stress function $v\in H^2(B_R)$, the linearized elastic energy is defined as
\begin{equation*}\label{eq:G_v}
\mathcal{G}(v;B_R)
\coloneqq \frac{1}{2}\frac{1+\nu}{E}\int_{B_R}\big(|\nabla^2 v|^2 - \nu\,(\Delta v)^2\big)\,\mathrm{d}x.
\end{equation*}
The two isolated disclinations in $B_R$ are modeled through a superposition of Dirac masses supported at the defects locations $y^1,y^2\in B_R$; more precisely, we introduce the \textit{disclination measure}
\begin{equation*}\label{disclination measure}
    \theta\coloneqq s_1 \, \delta_{y^{1}} + s_2 \, \delta_{y^{2}},
\end{equation*}
where the $s_k$'s are the associated Frank angles.
We combine the elastic energy~$\mathcal{G}$ and the measure~$\theta$ into the functional introduced in \cite[Formula (2.1)]{CDLM24}
\begin{equation*}\label{eq:I_theta}
\mathcal{I}^{\theta}(v;B_R) \coloneqq 
\frac{1}{2}\,\frac{1+\nu}{E}\,\int_{B_R}\big(|\nabla^2 v|^2 - \nu\,(\Delta v)^2\big)\,\mathrm{d}x + s_1 v(y^{1})+s_2 v(y^{2}),
\end{equation*}
where the last two summands are the explicit expression of the duality $\langle \theta, v\rangle$.
The corresponding Euler--Lagrange equation under traction-free boundary conditions reads
\begin{eqnarray}\label{2511111717} 
 \begin{cases}
\displaystyle
\displaystyle
\frac{1-\nu^2}{E}\,\Delta^2 \, v = -\theta & \text{in }  B_R, \\[2mm]
 \nabla^2\, v\, t = 0 & \text{on } \partial B_R,
\end{cases}
   \end{eqnarray}
and it has been shown in \cite[Proposition 1.10]{CDLM24} that
\begin{equation*}\label{equivalenza CB}
    \nabla^2\, v\, t = 0 \quad \text{on $\partial B_R$}
    \quad\text{if and only if}\quad 
    v = a, \quad \partial_n\,v = \partial_n\, a\quad \text{on $\partial B_R$,}
\end{equation*}
for some affine function $a\colon\mathbb{R}^2\to\mathbb{R}$. 
Since affine functions are in the kernel of the operators appearing in \eqref{2511111717}, without loss of generality we can choose $a=0$, thus reducing the traction-free boundary conditions to homogeneous boundary condition. The differential problem in \eqref{2511111717} now reads
\begin{eqnarray}\label{eq:biharmonic_theta} 
 \begin{cases}
\displaystyle
\frac{1-\nu^2}{E}\,\Delta^2\, v = -\theta & \text{in $B_R$\,,} \\[2mm]
\displaystyle v = \partial_n \,v = 0 & \text{on $\partial B_R$\,,}
\end{cases}
\end{eqnarray}
which is the Euler--Lagrange equation associated with the minimization problem
\begin{equation}\label{prob_min_v}
\min\bigl\{ \mathcal{I}^\theta(v;B_R) : v\in H^2_0(B_R)\bigr\}.
\end{equation}
In the same spirit as \cite{CGMP2025}, letting $\bar{v}$ be the unique solution to \eqref{prob_min_v},
by Clapeyron's theorem, the minimal elastic energy can be expressed by
\begin{equation}\label{Clapeyron generica}
\displaystyle\mathcal{G}(\bar{v}; B_R) = -\frac{1}{2}\langle\theta,\bar{v}\rangle=-\frac12\bigl(s_1 \bar v(y^1)+s_2\bar{v}(y^{2})\bigr) \eqqcolon \mathcal{W}(y^1,y^2;B_R).
\end{equation}
The solution $\bar{v}$ corresponds to that of the \textit{clamped disk problem} \eqref{eq:biharmonic_theta} and it can be expressed via the associated Green's function, see, \emph{e.g.}, \cite{Nakai1978}. 
In particular, the solution reads 
 \begin{equation*}\label{def di v}\displaystyle
\bar{v}(x) = \sum_{k=1}^2 \bar{v}^k(x), 
\end{equation*}
where the functions $\bar{v}^k\colon B_R\to\mathbb{R}$ are defined by
\begin{align*}
\displaystyle
    \bar{v}^k(x) &= -CR^2s_k \, \biggl[\frac{|x - y^k|\,^2}{R^2}\, \log\biggl( \frac{|x - y^k|\,^2}{R^2} \biggr)
    + \biggl(1 - \frac{|x|^2}{R^2}\biggr) \biggl(1 - \frac{|y^k|\,^2}{R^2} \biggr) \notag\\
    &\qquad \qquad \qquad-  \frac{|x - y^k|\,^2}{R^2} \log\biggl( \frac{R^4 - 2 R^2 x \cdot y^k + |x|^2 |y^k|\,^2}{R^4} \biggr)\biggr],
\end{align*}
if $x\neq y^k$, and are extended by continuity to $x=y^k$, with    $\bar{v}^k(y^k) = -CR^2s_k \Bigl (1 - \frac{|y^k|\,^2}{R^2} \Bigl)^2$,
where we introduce the constant
\begin{equation*}\label{costante C}
    \displaystyle
C \coloneqq \frac{E}{1 - \nu^2}\frac{1}{16\pi}
\end{equation*}
collecting the mechanical parameters.
The function $\mathcal{W}(\cdot,\cdot;B_R)\colon B_R \times B_R\to\mathbb{R}$ provides the value of the minimal elastic energy at the equilibrium as a function of the positions of the disclinations, and is particularly suited to study their dissipative dynamics.

\section{Dissipative dynamics}\label{sec_dissipative}

In this section, we derive the equations of motion for a pair of disclinations with opposite Frank angle 
\begin{equation}\label{eq_Frank_angles}
s_1=-s_2\eqqcolon s>0.
\end{equation}
In this context, the total Frank angle vanishes in $B_R$\,, but if the two disclinations get arbitrarily close, then, upon suitably rescaling by the dipole length, the resulting system is kinematically and energetically equivalent to an edge dislocation (see, \textit{e.g.} \cite{CDLM24, Eshelby66}). In particular, the Burgers vector of the resulting edge dislocation has magnitude $s$ and points in the direction which is the $\pi/2$-rotation of that from $y^2$ to $y^1$, see \cite[page~88]{CDLM24}.

With our choice of Frank angles~\eqref{eq_Frank_angles}, the explicit expression of the minimal elastic energy~\eqref{Clapeyron generica} is 
\begin{equation}\label{energia_elastica_dipolo}
\begin{split}
\mathcal{W}(y^1,y^2;B_R)=&\frac{1}{2}CR^2s^2\biggl[  \biggl(1-\frac{\abs{y^1}^2}{R^2}\biggr)^2+\biggl(1-\frac{\abs{y^2}^2}{R^2}\biggr)^2 \\
& -2\biggl(1-\frac{\abs{y^1}^2}{R^2}\biggr)\biggl(1-\frac{\abs{y^2}^2}{R^2}\biggr) \\
& +2\frac{\abs{y^2-y^1}^2}{R^2}\log\biggl(\frac{R^4-2R^2y^1\cdot y^2+\abs{y^1}^2\abs{y^2}^2}{R^2\abs{y^2-y^1}^2}\biggr)\biggr],
\end{split}
\end{equation}
and, under the assumption that disclinations evolve according to the \textit{maximal dissipation criterion} \cite{CermelliGurtin99,CermelliLeoni06,CGMP2025, HudsonMorandotti2017}, the velocity of each defect is proportional to the steepest descent of the minimal elastic energy. 
This leads to the following system of ODEs:
\begin{equation}
    \label{sistema base}
    \begin{cases}
        \dot{y}^k(t) = -\lambda_k \nabla_{y^k} \,\mathcal{W}(y^1,y^2;B_R), \qquad \text{for $t>0$,}\\
        y^k(0) = y^{k,0}\in B_R\,,
    \end{cases}
\end{equation}
for $k=1,2$.
The parameters $\lambda_k$ appearing above are a mobility parameters with units $[\lambda_k]=[\mathrm{s\,kg^{-1}}]$.
Various mobility functions are known for screw dislocations, see, \emph{e.g.}, \cite{BlassFonsecaLeoniMorandotti15, BonaschiVanMeursMorandotti17, CannoneElHajjMonneauRibaud10,CermelliGurtin99}; here, we focus only on the isotropic case and therefore we consider $\lambda_1=\lambda_2= \lambda\in\mathbb{R}$.

By observing that the energy $\mathcal{W}$ in \eqref{energia_elastica_dipolo} rescales quadratically in $R$ by passing from $B_1$ to $B_R$\,, namely 
\begin{equation*}
\mathcal{W}(Ry^1,Ry^2;B_R)=R^2\mathcal{W}(y^1,y^2;B_1), \qquad\text{for $y^1,y^2\in B_1$\,,}
\end{equation*}
following \cite[Section 3.1]{CGMP2025}, it is convenient to recast \eqref{sistema base} in non-dimensional form. 
By means of the rescaling
\begin{equation}\label{rescaling}
\frac{y^k}{R}\mapsto y^k 
\quad\text{and}\quad 
C\lambda t\mapsto t, 
\end{equation} 
and by defining $W(\cdot,\cdot)\colon B_1\times B_1\to\mathbb{R}$ as
$$W(y^1,y^2)\coloneqq \frac{1}{CR^2}\mathcal{W}(Ry^1,Ry^2;B_R), \qquad\text{for $y^1,y^2\in B_1$,}$$
system \eqref{sistema base}
becomes, for $k=1,2$,
\begin{equation}
    \label{sistema riscalato}
    \begin{cases}
        \dot{y}^k(t) = -\nabla_{y^k} W(y^1,y^2), \qquad \text{for $t>0$,}\\
        y^k(0) = y^{k,0}\in B_1\,.
    \end{cases}
\end{equation}

We will focus on an arrangement of disclinations $y^1=(y^1_1,y^1_2),y^2=(y^2_1,y^2_2)\in B_1$ located on a diameter, which they never abandon by symmetry reasons.
Therefore, if $y^{1,0}_2=y^{2,0}_2=0$ and $y^{1,0}_1>y^{2,0}_1$, then it is not restrictive to consider that $y^1_2(t)=y^2_2(t)=0$ and $y^1_1(t)>y^2_1(t)$ for every $t>0$.
The ordering $y^1_1>y^2_1$ is maintained throughout the evolution since collision of disclinations occurs in infinite time (see, \emph{e.g.}, \cite[Section~3.3.1]{CGMP2025}, where the study is conducted for the even more symmetric case $y^{1,0}_1=-y^{2,0}_1$). 
Owing to these considerations, it is convenient to introduce the new variables 
\begin{equation*}
    \begin{split}
    h(t) =&\, h\bigl(y^1(t),y^2(t)\bigr)
    \coloneqq \abs{y^1(t)-y^2(t)}=y_1^1(t)-y_1^2(t), \\
    d(t)= &\, d\bigl(y^1(t),y^2(t)\bigr)
    \coloneqq \frac12\bigl(y_1^1(t)+y_1^2(t)\bigr),
    \end{split}
\end{equation*}
which describe the disclination separation (dipole length, if it is very small) and the midpoint of the pair (center of the dipole, if $h$ is very small), respectively.
The relationship
\begin{equation}\label{cambiamento variabili}
y^k(t)=d(t)+\frac{(-1)^{k+1}}{2} h(t),\qquad\text{for $k=1,2$,}
\end{equation}
returns the positions of the disclinations starting from the center and the inter-disclination distance.

The dynamics is defined for $y^1_1\in(-1,1)$ and $y^2_1\in(-1,y^1_1)$ and it corresponds in a one-to-one fashion, to $d\in(-1,1)$ and $h\in(0,2)$ with the constraint that $d\pm\frac12 h\in(-1,1)$.
Should $d$ be negative, we could switch signs to $y^1_1$, $y^2_1$ and $s$, and bring ourselves back to the situation where $d\geq0$. 
Therefore, the dynamics for the pair $(h,d)$ is considered in the set 
\[
\mathcal{R} \coloneqq \Bigl\{(h,d) \in (0,2) \times [0,1) \colon 0 < d + \frac{h}{2} < 1\Bigr\},
\]
and it is determined, from \eqref{sistema riscalato}, by
\begin{equation}\label{sistema distanza-baricentro}
    \begin{cases}
        \displaystyle \dot{h} = 4s^2h\biggl[2 \log \biggl(\frac{4h}{4 - 4d^2 + h^2}\biggr)+\frac{4 - 4d^2 - h^2}{4 - 4d^2 + h^2}-2d^2\biggr]\eqqcolon f(h,d), \\[3mm]
        \displaystyle \dot{d} = \frac{2s^2h^2d\,(4d^2-h^2)}{4 - 4d^2 + h^2}\eqqcolon g(h,d),\\[3mm]
       \displaystyle h(0) = y^{1,0}_1-y^{2,0}_1 \eqqcolon h_0\in(0,2),\\[2mm]
       d(0) =  \frac12\bigl(y^{1,0}_1+y^{2,0}_1\bigr) \eqqcolon d_0\in[0,1),
    \end{cases}  
\end{equation}
with the constraint that $(h_0,d_0)\in\mathcal{R}$.

System \eqref{sistema distanza-baricentro} can more conveniently be re-written by introducing $\zeta\colon [0,+\infty)\to\mathbb{R}^2$ defined by $t\mapsto \zeta(t)\coloneqq (h(t),d(t))$, and the vector field $\Phi\colon\overline{\mathcal{R}}\to\mathbb{R}^2$ defined by $\Phi(h,d)\coloneqq(f(h,d),g(h,d))$ in $\mathcal{R}$, and extended by continuity on $\partial\mathcal{R}$ by
\begin{equation}\label{estensione}
    \begin{cases}
    \Phi(0,d) = (0,0), \\
    \displaystyle \Phi\Bigl(h,1-\frac{h}{2}\Bigr) = \big( -2s^2h\,(1-h)(2-h),s^2h\,(1-h)(2-h) \big).
    \end{cases}
\end{equation}
Thus, the dynamics~\eqref{sistema distanza-baricentro} can be extended to
\begin{equation}
    \begin{cases}\label{sistema distanza-baricentro nuovo}
        \dot{\zeta}=\Phi(\zeta),\\
        \zeta(0) = (h_0\,,d_0)\eqqcolon\zeta_0\in\overline{\mathcal{R}}.
    \end{cases}
\end{equation}

The next result ensures that the dynamics \eqref{sistema distanza-baricentro nuovo} is well posed for any initial datum $\zeta_0\in\overline{\mathcal{R}}\setminus\mathcal{L}$, where $\mathcal{L}\coloneqq \partial\overline{\mathcal{R}}\cap\{h=0\}$. 
Notice that the case in which $\zeta_0\in\mathcal{R}$ describes a physical situation with both disclinations inside the domain $B_1$\,; the case in which $\zeta_0=(h_0,1-h_0/2)$ for $h_0\in(0,2)$, corresponds to the case of $y_1^1=1\in\partial B_1$ and $y_1^2\in(-1,1)$, the disclination $y^1$ sitting at the boundary and the disclination $y^2$ being alone in the domain.
As we will see, when one disclination is at the boundary, it does not affect the dynamics of the remaining one.
Yet, considering this case is necessary to establish the well-posedness result (see Theorem~\ref{thm_wellposedness} below).

The case in which $\zeta_0=(0,d_0)$, for a certain $d_0\in[0,1]$ will be commented on in Remark~\ref{remark_32} below.

\begin{theorem}\label{thm_wellposedness}
For any initial datum $\zeta_0\in\overline{\mathcal{R}}\setminus\mathcal{L}$, problem~\eqref{sistema distanza-baricentro nuovo} admits a unique solution $\zeta\in C^1([0,+\infty);\overline{\mathcal{R}}\setminus\mathcal{L})$, satisfying $\zeta(0)=\zeta_0$\,.
\end{theorem}
\begin{proof}
We start by observing that $\Phi$ is locally Lipschitz-continuous on $\overline{\mathcal{R}} \setminus \mathcal{L}$.
Therefore, by standard results on ODE theory, for every choice of $\zeta_0 \in \overline{\mathcal{R}} \setminus \mathcal{L}$, there exists a unique local solution to \eqref{sistema distanza-baricentro nuovo}.

If $\zeta_0=(h_0,0)$, for some $h_0\in(0,2)$, then the dynamics is defined for all $t>0$ and, in this case, $d(t)=0$. 
This corresponds to the case of $y^1(t)=-y^2(t)$ for every $t>0$ and it has been thoroughly studied in \cite[Section~3.3]{CGMP2025}.

If $\zeta_0=(h_0,1-h_0/2)$, for some $h_0\in(0,2)$, then the dynamics corresponds to an initial condition for which the disclination $y^1$ starts off at the boundary of the domain and the disclination $y^2$ starts off in the interior of $B_1$\,. In this situation, since 
$$\frac{g(h,1-h/2)}{f(h,1-h/2)}=-\frac12,$$
the dynamics is confined to the edge $\{(h,d)=(h,1-h/2)\}$, so that $y^1$ neither moves nor influences the dynamics of $y^2$, which, in turn, behaves as the only defect in the domain.
The resulting equation of motion for $y^2\eqqcolon \eta$ is
\begin{equation}\label{dinamica_y^2}
    \dot \eta=2s^2\eta(1-\eta^2),\qquad \eta_{0}=1-h_0\in(-1,1),
\end{equation}
which is well defined for all $t>0$, and is the same that has been found and studied in \cite[Section~3.2]{CGMP2025}.

If $\zeta_0\in\mathrm{int}\,\mathcal{R}$, then the maximal solution is defined for $t\in[0,+\infty)$ and is confined in $\mathrm{int}\,\mathcal{R}$. 
Indeed, if there existed a time $\tilde t$ such that $\zeta(\tilde t) \in \{d=0\}\cup\{(h,1-h/2)\}$, then uniqueness of the solution of the two previous cases would be violated. 
\end{proof}

\begin{remark}\label{remark_32}
We notice that $\Phi$ loses Lipschitz continuity on $\{h=0\}$, so that uniqueness is not guaranteed there. 
Yet, if $\zeta_0=(0,d_0)$, then $\zeta(t)=(0,d_0)$ for every $t>0$ is a constant solution, see~\eqref{estensione}; 
on the other hand, if $\zeta_0\in\mathcal{R}$, then the velocity field $\Phi$ becomes infinitesimal in a right neighborhood of $\mathcal{L}$, suggesting that solutions can hit the boundary segment $\{h=0\}$ in infinite time. 
As a matter of fact, we will prove this in Proposition~\ref{riassunto_regime_convergente}, and this corresponds to the two disclinations colliding. This behavior is very interesting for us. 
Indeed, in \cite{CDLM24} it has been proved that upon rescaling by (the square of) the separation distance,
a colliding dipole of disclinations is energetically equivalent to an edge dislocation. 
Therefore, this suggests that when $t\mapsto h(t)$ vanishes, the dynamics $t\mapsto d(t)$ of the center of the dipole converges, after a suitable time rescaling, to that of an edge dislocation; see Section~\ref{sec_rescaling}.
\end{remark}

\section{Stability of the dynamics}\label{analysis_dynamics}
In this section, we study the equilibria of~\eqref{sistema distanza-baricentro nuovo}.
We will identify three isolated equilibrium points and we will single out the boundary segment~$\mathcal{L}$ as made of equilibrium points, and we discuss their stability and their mechanical interpretation. 

Moreover, we find that there are two regimes for the pair of disclinations, namely one in which it breaks and the two defects move away from one another ($\zeta(+\infty)=(2,0)$, discussed in Section~\ref{section diverging}), and one in which the dipole collides ($\zeta(+\infty)=(0,d_\infty)$, for some $d_\infty\in(0,1)$, discussed in Section~\ref{section converging}).

We recall that an \emph{equilibrium point} for system~\eqref{sistema distanza-baricentro nuovo} is a point $\zeta_*\in\overline{\mathcal{R}}$ such that $\Phi(\zeta_*)=(0,0)$.
Moreover, we say that an equilibrium point $\zeta_*\in\overline{\mathcal{R}}$ is \emph{locally asymptotically stable} if there exists a neighborhood $\mathcal{N}$ of $\zeta_*$ in $\overline{\mathcal{R}}$ such that for every initial condition $\zeta_0\in\mathcal{N}$, the corresponding solution~$\zeta$ to~\eqref{sistema distanza-baricentro nuovo} satisfies that $\lim_{t\to+\infty} \zeta(t)=\zeta_*$\,.
We say that an equilibrium point $\zeta_*\in\overline{\mathcal{R}}$ is \emph{unstable} if the dynamics associated with an initial condition $\zeta_0$ close to $\zeta_*$ grows far from $\zeta_*$\,.

One way to deduce the asymptotic stability of an equilibrium point~$\zeta_*$ is to study the linearization of~\eqref{sistema distanza-baricentro nuovo} at $\zeta_*$: if the Jacobian matrix $J\Phi(\zeta_*)$ has (complex) eigenvalues with negative real part, then~$\zeta_*$ is (locally) asymptotically stable (see, \emph{e.g.}, \cite{HSD}).

\smallskip

We now characterize the set 
\begin{equation*}
\mathcal{Z}_{\Phi}\coloneqq \bigl\{(h,d)\in\overline{\mathcal{R}}: \Phi(h,d)=(0,0) \bigr\}=\mathcal{Z}_f\cap \mathcal{Z}_g
\end{equation*}
of equilibria of~\eqref{sistema distanza-baricentro nuovo}, where
\begin{equation}\label{eq_zeri}
    \mathcal{Z}_f\coloneqq \mathcal{A}\cup\mathcal{L}\cup\{(2,0)\}
    \quad\text{and}\quad
    \mathcal{Z}_g\coloneqq \mathcal{B} \cup \mathcal{L} \cup \mathcal{S}.
\end{equation}
In \eqref{eq_zeri}, $\mathcal{A}$ is the arc along which $\tilde{f}(h,d)\coloneqq f(h,d)/4s^2h$ vanishes; $\mathcal{B}=\{(h,0),h\in[0,2]\}$, and $\mathcal{S}=\{(h,h/2),h\in[0,1]\}$.
Thus,
\begin{equation*}\label{eq_ZPhi}
    \mathcal{Z}_{\Phi}=\mathcal{L}\cup \{E_1,E_2,E_3\},
\end{equation*}
the three isolated points being
\[
E_1 \coloneqq (h_*, 0),\qquad 
E_2 \coloneqq (1, \tfrac{1}{2}), \qquad 
E_3 \coloneqq (2, 0),
\]
where $h_*\approx 0.8$ is the unique root of the function $f$ restricted to $\mathrm{int}\,\mathcal{B}$, namely of the function $h\mapsto f(h,0)$ for $h\in(0,2)$.
\begin{remark}\label{osservazione}
   The point $E_1$ corresponds to the symmetric arrangement $y^1_1=-y^2_1\approx0.4$, which was found in \cite[Section~3.3.1]{CGMP2025} (where our~$h_*$ corresponds to $\Delta_{\mathrm{eq}}$) as the stationary configuration of a disclination pair placed symmetrically with respect to the center of the disk. 
   
   The point $E_2\in \{(h,1-h/2)\}\cap\partial\mathcal{R}$ is the configuration where the disclination $y^1$ lies on the boundary, while the disclination $y^2$ sits at the center of the domain. 
   This corresponds to the equilibrium point $\eta=0$ for the corresponding dynamics~\eqref{dinamica_y^2}, which is that of an isolated disclination in $B_1$\,, see \cite[Section~3.2]{CGMP2025}. 
   Following the same interpretation, the point $E_3$ corresponds to the limiting situation of both disclinations $y^1$ and $y^2$ being located at antipodal points of $\partial B_1$\,. 
   Notice that $E_3$ corresponds to the equilibrium point $\eta=-1$ for~\eqref{dinamica_y^2}.
   In this case, no dynamics occurs.
\end{remark}

The following proposition characterizes the nature of the isolated equilibrium points $E_i$, $i=1,2,3$.
\begin{prop}[Characterization of the stationary configurations]\label{caratterizzazione punti stazionari}
Let $E_i$\,, $i = 1,2,3$, be the isolated equilibrium points of~\eqref{sistema distanza-baricentro nuovo}. 
Then $E_1$ and $E_2$ are unstable equilibria and $E_3$ is a locally asymptotically stable equilibrium.
\end{prop}
\begin{proof} 
By computing the Jacobian matrix of the vector field $\Phi$ and evaluating it at the points $E_i$\,, $i=1,2,3$, we find
$$J\Phi(E_1)=
\begin{pmatrix}
    \displaystyle\frac{256s^2}{(4 + h_*^2)^2}-8s^2 & 0\\ 
    0 & \displaystyle -\frac{2 h_*^4s^2}{4 + h_*^2}
\end{pmatrix},$$
which is already diagonal and with negative determinant, thus $E_1$ is unstable;
the eigenvalues of 
$$J\Phi(E_2)=
\begin{pmatrix}
    s^2 & -2s^2\\
    -s^2/2 & s^2
\end{pmatrix},$$
are $0$ and $2s^2$, so that $E_2$ is unstable; finally
$J\Phi(E_3)=-4s^2 I_{2\times 2}$ proves the local asymptotic stability of $E_3$\,.
\end{proof}

To study the dynamics out of these equilibrium points, as well as to investigate the stability of the set $\mathcal{L}$, it is convenient to single out the following three regions contained in $\mathrm{int}\,\mathcal{R}$ (see Figure~\ref{fig:regioni}):
\begin{align*}
\mathcal{R}_1 &\coloneqq \bigl\{(h,d)\in \overline{\mathcal{R}}\colon f( h,d)<0 \text{ and } g(h,d)> 0 \bigr\} ,\\
\mathcal{R}_2 &\coloneqq \bigl\{(h,d)\in \overline{\mathcal{R}} \colon f(h,d)<0 \text{ and } g(h,d)<0 \bigr\},\\
\mathcal{R}_3 &\coloneqq \bigl\{ (h,d)\in \overline{\mathcal{R}}\colon f(h,d)> 0 \text{ and } g(h,d)<0 \bigr\}.
\end{align*}
We notice that $\mathcal{A}=\partial\mathcal{R}_2\cap\partial\mathcal{R}_3$ and that $\mathcal{S}=\partial\mathcal{R}_1\cap\partial\mathcal{R}_2$\,.
\begin{figure}[h]
    \centering
    \includegraphics[width=1\linewidth]{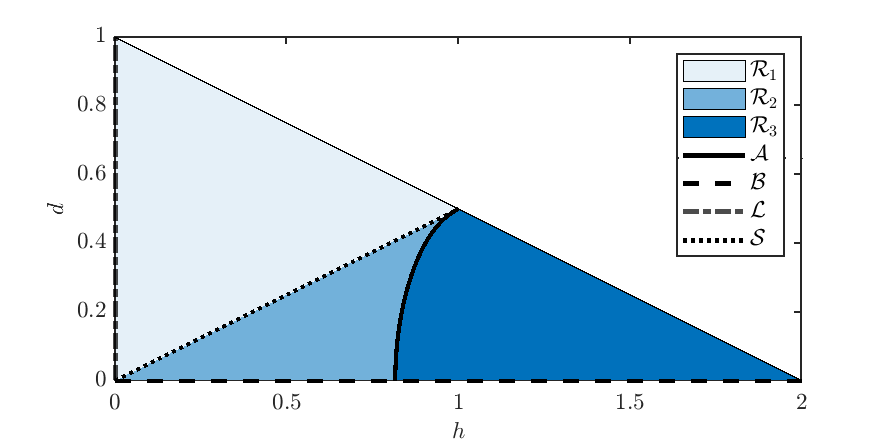}
    \caption{Plot of the regions $\mathcal{R}_i$\,, $i=1,2,3$.}
    \label{fig:regioni}
\end{figure}

We will prove that if $\zeta_0\in\overline{\mathcal{R}}_3\setminus\{E_1,E_2\}$, then the dynamics will converge to $E_3$ (see Proposition~\ref{AimplicaE3}), whereas, if $\zeta_0\in\overline{\mathcal{R}_1\cup\mathcal{R}_2}\setminus\mathcal{A}$, then the dynamics will converge to a point in $\mathcal{L}$ (see, Proposition~\ref{riassunto_regime_convergente}).
In other terms, the arc $\mathcal{A}$ divides the dynamics of the pair in a diverging regime, characterized by $\zeta(+\infty)=E_3$, and in a converging regime, characterized, instead, by $\zeta(+\infty)\in\mathcal{L}$.

\subsection{Diverging disclinations regime}\label{section diverging}
We now analyse the region where the disclinations move toward the boundary along opposite direction. 
In the $(h,d)$ plane, this is equivalent to proving that the dynamics tends to the equilibrium point $E_3$\,.

We start by proving a structural property of the arc $\mathcal{A}_*$, with $\mathcal{A}_* \coloneqq \mathcal{A}\setminus\{E_1,E_2\}$, namely that it is the graph of a function.
\begin{lemma}\label{A_is_a_graph}
    There exists a strictly increasing function $d_{\mathcal{A}_*}\colon (h_*,1)\to(0,1/2)$ such that $\mathcal{A}_*=\mathrm{graph}(d_{\mathcal{A}_*})=\{(h,d)\in\mathcal{R}: h\in(h_*,1), d=d_{\mathcal{A}_*}(h)\}$.
\end{lemma}
\begin{proof}
    The result follows from an application of the Implicit Function Theorem.
    Indeed, it is easy to verify that $\partial_d \tilde{f}(h,d)<0$ for every $(h,d)\in\mathcal{A}_*$. 
    Therefore, the condition $\tilde{f}(h,d)=0$ can be inverted on $\mathcal{A}_*$ and this proves that there exists $d_{\mathcal{A}_*}\colon (h_*,1)\to (0,1/2)$ such that $\mathcal{A}_*=\mathrm{graph}(d_{\mathcal{A}_*})$. 
    The strict monotonicity of $d_{\mathcal{A}_*}$ follows from implicit differentiation:
    $$d_{\mathcal{A}_*}'(h)=-\frac{\partial_h\tilde{f}(h,d_{\mathcal{A}_*}(h))}{\partial_d\tilde{f}(h,d_{\mathcal{A}_*}(h))},$$
    and an explicit verification shows that $d_{\mathcal{A}_*}'(h)>0$ for every $h\in(h_*,1)$.
\end{proof}

We are now ready to prove the main result in this section.
\begin{prop}[Diverging regime]\label{AimplicaE3}
    For every $\zeta_0\in\overline{\mathcal{R}}_3\setminus\{E_1,E_2\}$, the unique solution to~\eqref{sistema distanza-baricentro nuovo} is such that 
    \begin{equation}\label{E3assorbe}
    \lim_{t\to+\infty} \zeta(t)=E_3\,.
    \end{equation}
\end{prop}
\begin{proof}
    Owing to Theorem~\ref{thm_wellposedness}, the dynamics \eqref{sistema distanza-baricentro nuovo} is well posed.
    By Lemma~\ref{A_is_a_graph}, if $\zeta_0\in \mathcal{A}_*$\,, then $\dot{\zeta}(0)=\Phi(\zeta_0)=(0,g(h_0,d_0))$, with $g(h_0,d_0)<0$. 
    Thus, $\zeta(t)$ enters the region $\mathrm{int}\,\mathcal{R}_3$ as soon as~$t$ is positive.
    Once $\zeta(\bar t)\in\mathcal{R}_3$ for a certain $\bar t\geq0$, then $\zeta(t)\in\mathcal{R}_3$ for every $t\geq\bar t$.
    By uniqueness and by the monotonicity of the components of $\Phi$ in $\mathcal{R}_3$\,, \eqref{E3assorbe} follows.
\end{proof}

\begin{remark}
Reading Proposition~\ref{AimplicaE3} for the dynamics of the two disclinations yields that, for large times, they behave as two isolated defects, whose motion is only due to the attraction to the boundary.
Indeed, by considering the leading order of~\eqref{sistema distanza-baricentro nuovo} for $\zeta\to E_3$\,, \emph{i.e.}, $h\to2$ and $d\to0$, we have (see \eqref{sistema distanza-baricentro})
\begin{equation*}\label{divergente_semplificato}
    \begin{cases}
        \displaystyle \dot{h} = 4s^2(2-h)\\
        \dot{d}=-4s^2d\\
        h(0)=h_0\,, \, d(0) = d_0
    \end{cases}
\end{equation*}
(we can assume $d_0\ll1$ without loss of generality), whose solution is
\begin{equation*}\label{dinamica_divergente}
    h(t) = 2-(2-h_0) \exp(-4s^2t),
    \quad\text{and}\quad
    d(t) = d_0 \exp(-4s^2 t).
\end{equation*}
Recalling~\eqref{cambiamento variabili}, one immediately sees that the dynamics of the defects is
\begin{equation*}
    y^k(t)=(-1)^{k+1}-((-1)^{k+1}-y^{k,0})\exp(-4s^2t),
\end{equation*}
which is consistent with the results in \cite[Section~3.2]{CGMP2025}, where an isolated disclination tends to the boundary of the domain.
\end{remark}

\subsection{Converging disclinations regime}\label{section converging}
Let us focus on the region where the disclination pair tends to collapse, namely when $\zeta \to (0, d_\infty)\in\mathcal{L}$ for some $d_\infty\in[0,1]$ as $t\to +\infty$. 

We start by proving the following proposition, which, in the language of dynamical systems, states that the line $\mathcal{L}$ is the $\omega-$limit set of the region $\overline{\mathcal{R}_1 \cup \mathcal{R}_2}\setminus\mathcal{A}$.

\begin{prop}\label{R2_in_R1}
    For any initial datum $\zeta_0 \in \overline{\mathcal{R}_1 \cup \mathcal{R}_2}\setminus\mathcal{A}$, the corresponding unique solution to \eqref{sistema distanza-baricentro nuovo} is such that $\zeta(+\infty) \in \mathcal{L}$.
\end{prop}
\begin{proof}
By Theorem~\ref{thm_wellposedness}, the solution to~\eqref{sistema distanza-baricentro nuovo} is well defined for all positive times and is confined in the region~$\mathcal{R}$.
    Since $\zeta_0$ is not in the basin of attraction of the equilibrium point $E_3$ (see Proposition~\ref{AimplicaE3}), $\zeta(+\infty)\neq E_3$\,. 
    If $\zeta(+\infty)=\bar\zeta\in\mathcal{R}_1\cup\mathcal{R}_2$\,, this would mean that there exists an asymptotically stable equilibrium in the interior of $\mathcal{R}$, contradicting our previous analysis.
    Therefore, it must be that $\zeta(+\infty)\in\mathcal{L}$.
\end{proof}

In view of Proposition~\ref{R2_in_R1}, we now focus on a right neighbourhood $U_0\subset\mathcal{R}_1$ (to be specified presently) of $\mathcal{L}$.
In $U_0$\,, since $h\approx0$, we may analyze system \eqref{sistema distanza-baricentro nuovo} by considering only the leading order in $h$ for each variable, namely
\begin{equation}\label{accoppiato}
    \begin{cases}
        \displaystyle \dot h = 8s^2h\log\frac{h}{1-d^2}\eqqcolon\varphi(h,d)\,, \\[2mm]
        \displaystyle \dot d =2s^2h^2\frac{d^3}{1-d^2}\eqqcolon\gamma(h,d)\,, \\[2mm]
        h(0)=h_0, \,d(0)=d_0\,,
    \end{cases}
\end{equation}
with $\zeta_0=(h_0, d_0)\in U_0$\,.
System \eqref{accoppiato} cannot be integrated explicitly, yet existence and uniqueness of solutions to the Cauchy problem can be obtained by an application of the standard results from ODE theory.
In the regime $h\approx0$, the asymptotic dynamics in~\eqref{accoppiato} exhibits a fast-slow structure. Indeed, $\gamma$ is negligible with respect to $\varphi$ as $h\to 0$, determining that $d$ evolves on a slower time scale than $h$.

In view of the increasing monotonic behavior and the concavity of the curve of $h\mapsto d(h)$,\footnote{As a consequence of the Implicit Function Theorem, we have that 
$$\frac{\mathrm{d}^2(d(h))}{\mathrm{d}h^2}  =\frac{\mathrm{d}}{\mathrm{d}h} \biggl(\frac{\gamma(h,d(h))}{\varphi(h,d(h))}\biggr)= \frac{d^3}{4(1-d^2)\log(h/(1-d^2))}+o(h),\qquad\text{as $h\to0$,}$$ 
which is, to leading order in $h$, negative.} we can estimate that $d_\infty\in(d_0, d_\mathrm{max})$, where 
\begin{equation*}\label{eq_dmax}
d_{\mathrm{max}}\coloneqq d_0\Bigl(1-\frac{h_0^2d_0^2}{4(1-d_0^2)\log( h_0/(1-d_0^2))}\Bigr)\,,
\end{equation*}
and corresponds to the evolution along the tangent line with slope determined by the initial conditions, namely $\gamma(\zeta_0)/\varphi(\zeta_0)$ (see Figure~\ref{fig:tangente}).
\begin{figure}[h]
    \centering
    \includegraphics[width=.9\linewidth]{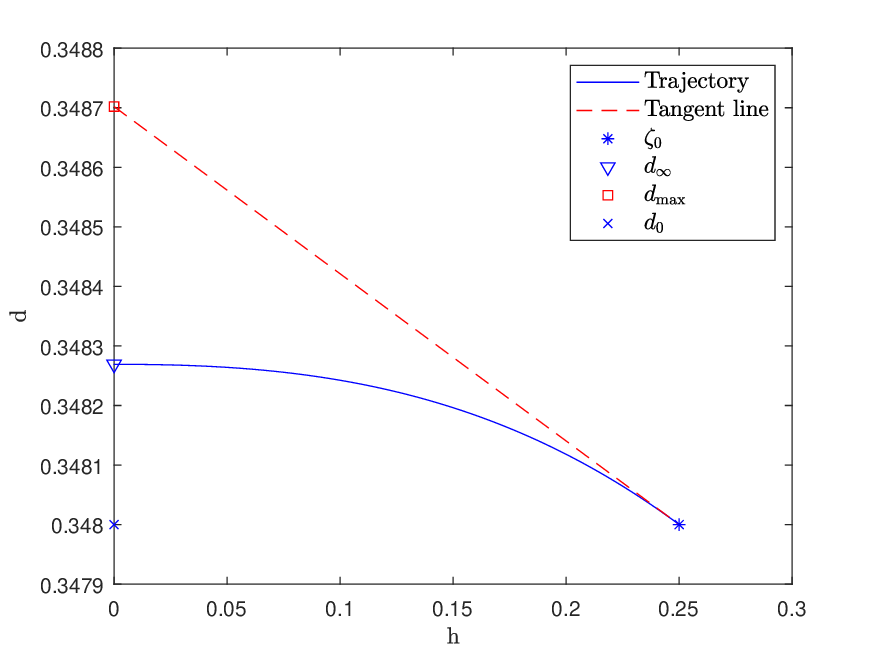}
    \caption{Plot of the trajectory $h\mapsto d(h)$ (blue line), and its tangent line (red dotted line). }
    \label{fig:tangente}
\end{figure}

From the explicit expression of $d_{\mathrm{max}}$\, the region $U_0$ can be determined as the one where the initial condition $\zeta_0$ satisfies: 
\[
\zeta_0\in \mathcal{R}_1\,, 
\quad 
h_0<1-d_{\mathrm{max}}^2\,, 
\quad\text{and}\quad 
d_{\mathrm{max}}<1.
\]

We are now ready to prove the main result of this section.
\begin{prop}[Converging regime]\label{riassunto_regime_convergente}
Under the same assumption of Proposition~\ref{R2_in_R1}, the unique solution to \eqref{sistema distanza-baricentro nuovo} is such that $h\in L^2(\mathbb{R}^+)$.
\end{prop}
\begin{proof}
As a consequence of Proposition~\ref{R2_in_R1}, for every $\zeta_0\in \mathcal{R}_1\cup\mathcal{S}\cup \mathcal{R}_2$\, there exists a time $t^* \geq 0$ such that  $\zeta(t^*) \in U_0$\,, where we can use~ \eqref{accoppiato}, up to restarting the dynamics and setting $t^*=0$. In particular, due to the monotonicity of the dynamics (in $U_0$\,, $h$ is strictly decreasing and $d$ is strictly increasing), and that of $\varphi$ with respect to~$d$, we deduce that 
\[
\varphi(h,d_0)\le \varphi(h,d)\le \varphi(h,d_\mathrm{max}).
\]
By comparison, it follows that 
\(h_\ell(t)\le h(t)\le h_\mathrm{u}(t),\) for every $t>0$,
where~$h_\ell$ and~$h_{\mathrm{u}}$ are the solutions to 
\[
\begin{cases}
    \dot h = \varphi(h,d_0),\\
    h(0)=h_0
\end{cases}
\qquad\text{and}\qquad
\begin{cases}
    \dot h = \varphi(h,d_{\mathrm{max}}),\\
    h(0)=h_0
\end{cases}
\]
respectively, whose explicit solutions are
\[
\begin{split}
h_\ell(t)=& (1-d_0^2)\exp\biggl(\log\frac{h_0}{1-d_0^2}\cdot \exp(8s^2t)\biggr), \\
h_{\mathrm{u}}(t)=& (1-d_{\textrm{max}}^2)\exp\biggl(\log\frac{h_0}{1-d_{\textrm{max}}^2}\cdot \exp(8s^2t)\biggr).
\end{split}
\]
Since both $\log(h_0/(1-d_0^2))$ and $\log(h_0/(1-d_{\mathrm{max}}^2))$ are negative in $U_0$\,, it follows that $h_\ell\,, h_{\mathrm{u}}\in L^2(\mathbb{R}^+)$, and the thesis follows by comparison.
\end{proof}

\section{Macroscopic time rescaling and effective dynamics}\label{sec_rescaling}

Motivated by the equivalence between an edge dislocation and a collapsing disclination dipole, first identified at the kinematic level by Eshelby \cite{Eshelby66} and subsequently established at the energetic level in \cite{CDLM24}, we show that the dynamics of the dipole center is governed by the same evolution law as that of an edge dislocation. This extends the equivalence to the dynamical setting.

This  is obtained by a suitable time rescaling, that is necessary because of the separation of time scales in~\eqref{accoppiato}.

By reformulating the dynamics according to the time rescaling \eqref{def tau}, the collision of the dipole now occurs in finite time. This behavior is analogous, at least heuristically, to that observed for dislocations. 
In Remark~\ref{remark_below} we comment that  this analogy is in fact exact. 
We point out that the $L^2$ integrability of~$h$ plays a crucial role here.

\begin{theorem}[Effective dislocation dynamics]\label{time rescaling}
Let $\zeta_0 \in U_0$ and let $t\mapsto h(t)$ be the associated solution to the Cauchy problem~\eqref{accoppiato}. 
Let $\tau \colon \mathbb{R}^+\to \mathbb{R}^+$ be (the macroscopic time) defined as
\begin{equation}\label{def tau}
\tau(t)\coloneqq \frac12\int_0^t h^2(s)\,\mathrm{d}s.
\end{equation}
Then the function $\tau$ is invertible and the rescaled dynamics of $d$ as a function of $\tau$, namely $\delta(\tau)\coloneqq d(\sigma(\tau))$ ($\sigma$ being the inverse function of $\tau$) satisfies
\begin{equation}\label{soluzione riscalata}
\displaystyle
    \log\Bigl(\frac{\delta^2}{{\delta_0}^2}\Bigr) + \frac{1}{\delta^2} - \frac{1}{\delta_0^2} + 8s^2\tau = 0.
\end{equation}   
Moreover, $\delta$ collides with the boundary of the domain in finite time.
\end{theorem}
\begin{proof}
The function $\tau$ defined in \eqref{def tau} has the following properties: $\tau(0)=0$, it is strictly increasing, and
\[\tau_\infty\coloneqq\lim_{t\to+\infty} \tau(t)=\frac12 \int_0^{+\infty} h^2(s)\,\mathrm{d}s \in\mathbb{R}^+,
\]
owing to the fact that $h\in L^2(\mathbb{R}^+)$ (see Proposition~\ref{riassunto_regime_convergente}).

The inverse function $\sigma$ is therefore defined from the interval $[0,\tau_\infty]$ to $\mathbb{R}^+$, and allows us to define the rescaled dynamics of the centre of the dipole via $\delta(\tau)= d(\sigma(\tau))$.
By the chain rule, the dynamics for $d$ in \eqref{accoppiato} reads now
\begin{equation}\label{dinamica_riscalata}
    \begin{cases}
    \displaystyle \delta'(\tau) = 4s^2\frac{\delta^3(\tau)}{1 - \delta^2(\tau)}  & \text{for $\tau\in(0,\tau_\infty]$,}\\
    \delta(0)=\delta_0\coloneqq d_0\,.
    \end{cases}
\end{equation}
By separation of variables, the solution to \eqref{dinamica_riscalata} is given by \eqref{soluzione riscalata}. 
Moreover, since $\delta'(\tau)>0$ for all $\tau$, the motion occurs towards the boundary and the time needed to reach it ($\delta =1$) is
\[\displaystyle
\bar{\tau} = \frac{1}{8s^2}\bigg(\log{\delta_0}^2-1+\frac{1}{{\delta_0}^2}\bigg)\,.\qedhere
\]
\end{proof}

\begin{remark}\label{remark_below}
    The time rescaling introduced by~\eqref{def tau} allows us to appreciate the dynamics of the $d$ variable, which is much slower than that of the dipole length~$h$ in~\eqref{accoppiato}.
    In particular, owing to the introduction of the macroscopic time $\tau$ and by recasting the dynamics of~$d$ in \eqref{accoppiato} as in \eqref{dinamica_riscalata}, we see that the equation for~$\delta$ coincides with that of a collapsing dipole of disclinations found in \cite[Section~3.2]{CMS}.
    
There, the dipole length $h$ is treated as a small parameter that converges to zero.
Here, on the contrary, the function $t\mapsto h(t)$ is a solution to the dynamics. In the converging disclinations regime treated in Section~\ref{section converging}, it vanishes. This suggests the question whether these two limiting cases may lead to the same mechanical behaviour.
Theorem~\ref{time rescaling} proves that the two resulting dynamics (that is \cite[Section~3.2]{CMS} and that in \eqref{dinamica_riscalata}) are the same. 

This may be thought of as a dynamical version of Eshelby equivalence: a moving colliding dipole of disclinations generates a moving edge dislocation. 
Moreover, if this dislocation is not placed at the center of the disk (as is our case), then it is attracted to the boundary and collides on it in finite time (see also \cite[Section~4.2]{HudsonMorandotti2017} for the analogous case of a screw dislocation). 
\end{remark}

\section{Numerical solutions and phase diagram}\label{phase diagram}
In this section, we describe some numerical experiments that we ran to clarify the dynamics in~$\mathcal{R}$.
Let us consider the time interval $[0, T]$, for some $T>0 $. Given $N\in \mathbb{N}\setminus \{0\}$ and $n\in \{0,\dots,N\}$, we define 
\[\Delta t\coloneqq \frac{T}{N}, \qquad \text{and}\qquad t_n\coloneqq n\Delta t =\frac{n}{N} T\]
and then introduce 
\[
h^n\coloneqq h(t_n) \qquad \text{and} \qquad d^n\coloneqq d(t_n), 
\]
with $ h^0=h_0$ and $d^0=d_0$\,. 
Therefore, the time-discrete version of system \eqref{sistema distanza-baricentro nuovo} in terms of $\zeta^n = (h^n,d^n)$ reads
\[\zeta ^n = \zeta^ {n-1}+\Delta t \, \Phi(\zeta^{n-1}), \qquad \text{with $\zeta^0=(h_0,d_0)$,}\]
which can be solved via the explicit Euler scheme.
We implemented our own formulation of the explicit Euler method and run it in the MATLAB R2026b environment. 

We consider different initial configurations, in order to capture the emergence of the possible behaviors of a pair of disclinations arranged radially. 
The results of the simulations are shown in Figure~\ref{fig:analisi qualitativa numerica} and are divided into diverging (on the top) and converging (on the bottom) regimes.

\begin{figure}
    \centering
    \begin{subfigure}[b]{1\textwidth}
        \centering
        \includegraphics[width=\textwidth]{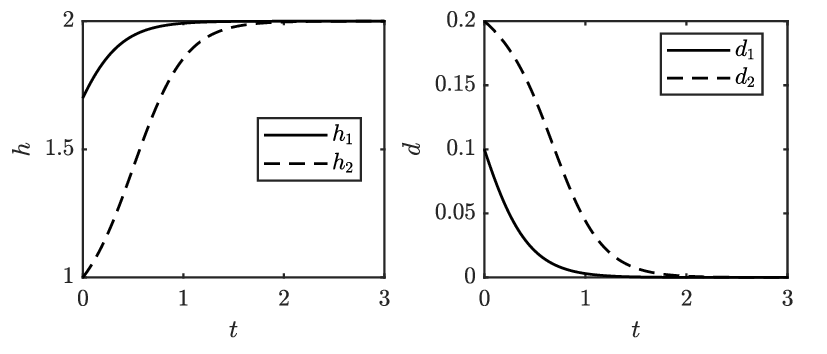}
        \caption{Diverging regime}
        \label{fig:divergente}
    \end{subfigure}
    \hfill
    \begin{subfigure}[b]{1\textwidth}
        \centering
        \includegraphics[width=\textwidth]{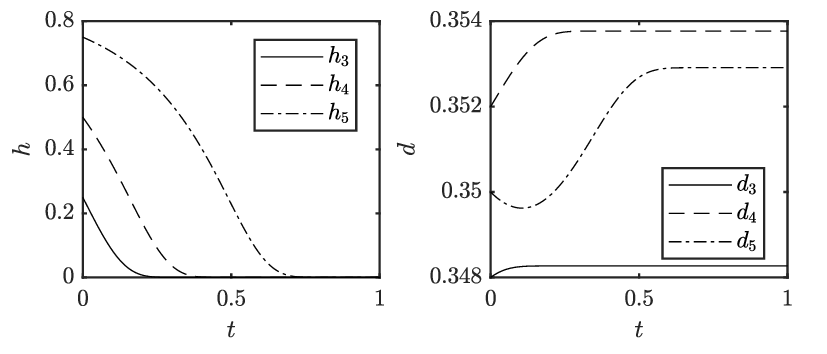}
        \caption{Converging regime}
        \label{fig:congergente}
    \end{subfigure}
    \caption{Possible scenarios of the dynamics of the pair of disclinations. On the top, the diverging regime characterized by $h\to 2$ and $d\to 0$. On the bottom, the converging regime characterized by $h\to 0$ and $d\to d_\infty$\,.}
    \label{fig:analisi qualitativa numerica}
\end{figure}

As pointed out in \cite[Section~3.1]{CGMP2025}, the dynamics of disclinations is the result of the interplay of three factors: (i) the presence of the boundary, (ii) the disclination-boundary interaction, and (iii) the pairwise interactions between disclinations. 
The identification of these three different contributions provides a useful tool for understanding the numerical simulations.
In particular, the diverging regime simulations reveal two qualitatively different scenarios. 
Referring to the top row of Figure~\ref{fig:analisi qualitativa numerica}, in the $(h_1, d_1)$ scenario, the monotonicity and the concavity properties of the trajectories of~$h$ and~$d$ suggest that the boundary effects dominate the evolution.
On the other hand, in the ($h_2,d_2$) scenario, the change in concavity observed in both trajectories indicates an initial competition between the mutual disclination interaction and the boundary contribution. 

By contrast, the converging regime is characterized by three qualitatively different scenarios. 
In the ($h_3,d_3$) scenario, the monotonicity and the concavity properties of the trajectories of $h$ and $d$ indicate that the mutual disclination interaction is dominant; while, in the ($h_4,d_4$) scenario, the monotonicity and the change in concavity observed in both trajectories suggest an initial competition between the mutual disclination interaction and the boundary contribution. 
Finally, in the ($h_5, d_5$) scenario, the behavior of the trajectories (in particular that of $d$) suggests an initial balance between the mutual disclination interaction and the boundary contribution, which in the end evolves as in the ($h_3,d_3$) scenario.

\smallskip

We conclude this section by constructing a map that associates to each initial point $\zeta_0=(h_0,d_0)\in \overline{\mathcal{R}}\setminus\mathcal{L}$ the final point $\zeta_{\text{end}}=(h_{\text{end}},d_{\text{end}})\in\overline{\mathcal{R}}\setminus\mathcal{L}$ determined by numerically solving \eqref{sistema distanza-baricentro nuovo} up to a time $\overline{T}>0$.
The phase diagram in Figure~\ref{fig:comportamento difetti} (top picture) is constructed by choosing $\overline{T}=20$ and $N=10^6$.

\begin{figure}
	\begin{minipage}{.9\textwidth}
    \centering
    \includegraphics[width=1\textwidth]{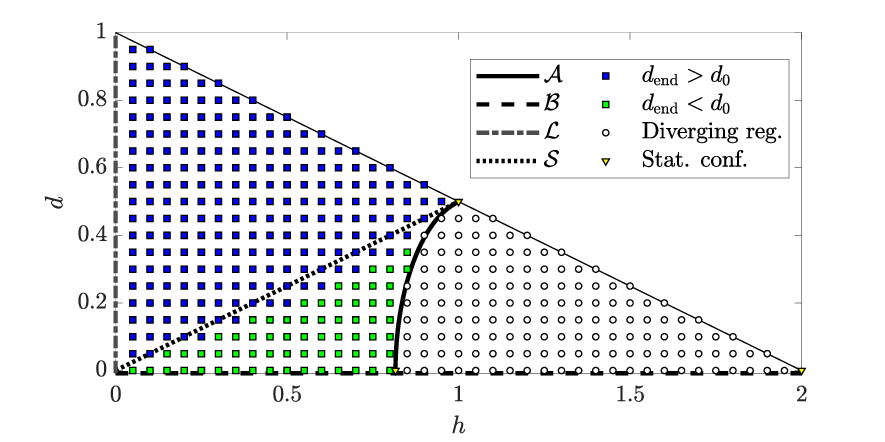}
    \end{minipage}
	\begin{minipage}{.9\textwidth}
     \includegraphics[width=1\textwidth]{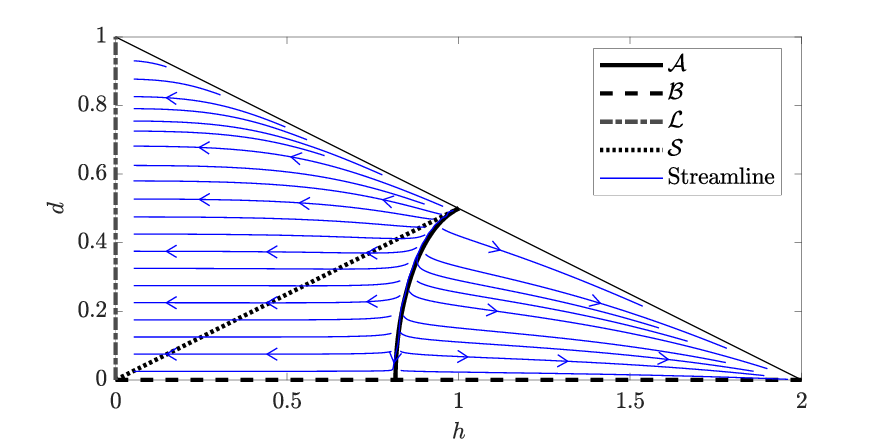}
     \end{minipage}
\caption{On the top, the phase diagram of the disclination dynamics.  
Here squares correspond to initial configurations in the colliding dipole regime while circles indicate initial configurations in the diverging disclination regimes. 
On the bottom, the flux lines plotted with the \textit{streamslice} native Matlab function.}\label{fig:comportamento difetti}
\end{figure}

We classify each initial point $\zeta_0$ in the domain in terms of its final state $\zeta_{\text{end}}$.
  
Investigation of the numerical solutions confirm the emergence of the three main regimes analyzed: converging regime, diverging regime, and stationary regime. In particular, 
\begin{itemize}
    \item The filled square represents the converging behavior: the defects are converging to $(0, d_{\mathrm{end}})\in\mathcal{L}$ for some point $d_{\mathrm{end}}\in [0, 1]$. 
    Blue squares indicate that the final position of the midpoint of the dipole has increased compared to the initial position ($d_\mathrm{end}>d_0$), while green squares show that the midpoint of the dipole has  decreased ($d_\mathrm{end}<d_0$).
    \item The circle represents the diverging behavior: the two defects are far enough not to experience the presence of the other defect and therefore, each one behaves as an isolated disclination and tends to reach the boundary of the domain ($h\to 2, \,d\to 0$).
    \item The triangle represents the stationary points. In particular, we note that, for the sake of completeness, we manually added the initial configuration sample the stationary configuration $E_1$.
\end{itemize}
Finally, we complete the phase diagram by using the Matlab native function \textit{streamslice} to draw the flux lines of the problem.
Figure \ref{fig:comportamento difetti} (bottom picture) clearly shows that the arc $\mathcal{A}$ is the border between the converging and the diverging regime.

\bigskip

\paragraph{\textbf{Acknowledgements}}
JASSO scholarship
offered within the “Kyushu University Program for Emerging Leaders in Science” (Q-PELS) and local support received at the Institute of Mathematics for Industry, an international Joint Usage and
Research facility located at Kyushu University (NB).
PC’s work is supported by JSPS KAKENHI Grant-in-Aid for Scientific Research (C) JP24K06797. 
%and  by JST A--STEP (Grant Number JPMJTR24T6). 
MM acknowledges partial support from the MUR grant Geometric Analytic Methods for PDEs and Applications (2022SLTHCE cup E53D23005880006). This manuscript reflects only the authors’ views and opinions and the Italian Ministry cannot be considered responsible for them.
\noindent PC holds an honorary appointment at La Trobe University. 
NB, PC, and MM are members of the Gruppo Nazionale per l’Analisi Matematica, la Probabilità e le loro Applicazioni (GNAMPA) of the Istituto Nazionale di Alta Matematica (INdAM). 
MM and NB thank the Institute of Mathematics for Industry, an International Joint Usage and Research Center located in Kyushu University; PC  thanks the Department of Mathematical Sciences ``G.~L.~Lagrange'' of Politecnico di Torino where part of the work contained in this paper was carried out.

\bibliographystyle{abbrv}
\bibliography{refsPatrick}
\end{document}